\documentclass[11pt,reqno,a4paper]{amsart}
\usepackage[utf8]{inputenc}		% Unicode support; allows entering special characters in a non-escaped form
\usepackage[T1]{fontenc}			% Display special characters from the T1 standard
\usepackage{hyperref}			% Clickable references; does not work with beamer!

\usepackage{amsmath,amsfonts}	% Bunch of math support
\usepackage{amsthm,latexsym,amssymb}
\usepackage{mathrsfs}
\usepackage{eucal}

\usepackage[all]{xy}
\newtheoremstyle{mytheoremstyle} % name
    {0.6cm}                    % Space above
    {0.6cm}                    % Space below
    {\itshape}                   % Body font
    {}                           % Indent amount
    {\bfseries}                   % Theorem head font
    {.}                          % Punctuation after theorem head
    {.5em}                       % Space after theorem head
    {}  % Theorem head spec (can be left empty, meaning ‘normal’)

\newtheoremstyle{mytheoremstarstyle} % name
    {0.6cm}                    % Space above
    {0.6cm}                    % Space below
    {\itshape}                   % Body font
    {}                           % Indent amount
    {\bfseries}                   % Theorem head font
    {.}                          % Punctuation after theorem head
    {.5em}                       % Space after theorem head
    {\thmname{#1}\thmnote{ (#3)}}  % Theorem head spec (can be left empty, meaning ‘normal’)
  
\newtheoremstyle{mydefinitionstyle} % name
    {0.6cm}                    % Space above
    {0.6cm}                    % Space below
    {\normalfont}                   % Body font
    {}                           % Indent amount
    {\bfseries}                   % Definition head font
    {.}                          % Punctuation after theorem head
    {.5em}                       % Space after theorem head
    {}  % Definition head spec (can be left empty, meaning ‘normal’)

\newtheoremstyle{mydefinitionstarstyle} % name
    {0.6cm}                    % Space above
    {0.6cm}                    % Space below
    {\normalfont}                   % Body font
    {}                           % Indent amount
    {\bfseries}                   % Definition head font
    {.}                          % Punctuation after theorem head
    {.5em}                       % Space after theorem head
    {\thmname{#1}\thmnote{ (#3)}}  % Definition head spec (can be left empty, meaning ‘normal’)
    
\newtheoremstyle{myremarkstyle} % name
    {0.6cm}                    % Space above
    {0.6cm}                    % Space below
    {\normalfont}                   % Body font
    {}                           % Indent amount
    {\bfseries}            	       % Remark head font
    {.}                          % Punctuation after theorem head
    {.5em}                       % Space after theorem head
    {}  % Remark head spec (can be left empty, meaning ‘normal’)      

\newtheoremstyle{myremarkstarstyle} % name
    {0.6cm}                    % Space above
    {0.6cm}                    % Space below
    {\normalfont}                   % Body font
    {}                           % Indent amount
    {\bfseries}            	       % Remark head font
    {.}                          % Punctuation after theorem head
    {.5em}                       % Space after theorem head
    {\thmname{#1}\thmnote{ (#3)}}  % Remark head spec (can be left empty, meaning ‘normal’)      

\theoremstyle{mytheoremstyle}
\newtheorem{theorem}{Theorem}%[section]

\newtheorem{lemma}[theorem]{Lemma}
\newtheorem{proposition}[theorem]{Proposition}

\theoremstyle{mytheoremstarstyle}  
\newtheorem{theorem*}{Theorem}
\newtheorem{corollary*}{Corollary}
\newtheorem{lemma*}{Lemma}
\newtheorem{proposition*}{Proposition}

\theoremstyle{mydefinitionstyle}\newtheorem{definition}[theorem]{Definition}
\theoremstyle{mydefinitionstarstyle}\newtheorem{definition*}{Definition}

\theoremstyle{myremarkstyle}\newtheorem{example}[theorem]{Example}\newtheorem{examples}[theorem]{Examples}
\theoremstyle{myremarkstarstyle}\newtheorem{remark*}{Remark}\newtheorem{example*}{Example}\newtheorem{examples*}{Examples}

\def\Hom{\mathrm{Hom}}

\def\im{\mathrm{\rm im}}

\def\bbZ{\mathbb{Z}}\def\bbR{\mathbb{R}}
\def\Z{\bbZ}\def\R{\bbR}

\def\incl{\mathrm{incl}}
\def\const{\mathrm{const}}
\def\id{\mathrm{\rm id}}
\def\pr{\mathrm{\rm pr}}

\def\sSet{\mathcal{S}\mathrm{et}_\Delta}

\def\Cat{\mathcal{C}\mathrm{at}}

\def\Top{\mathcal{T}\mathrm{op}}

\def\Ch{\mathcal{C}\mathrm{h}}

\def\scrC{\mathcal{C}}
\def\scrD{\mathcal{D}}
\def\scrE{\mathcal{E}}
\def\scrI{\mathcal{I}}

\usepackage{stmaryrd}
\usepackage{tikz}
\usepackage{ stmaryrd }
\usepackage[a4paper, left=3.7cm, right=3.7cm, top=3.5cm, bottom=3cm]{geometry}

\usetikzlibrary{matrix,arrows}

\newcommand{\xlongrightarrow}[1]{\xrightarrow{\hspace{1ex}#1\hspace{1ex}}}
\def\adj{\mathrm{adj}}

\def\pt{\mathrm{pt}}
\def\scrE{\mathcal{E}}
\def\ch{\mathrm{ch}}
\def\Map{\mathcal{M}\mathrm{ap}}
\def\TwoMon{\mathrm{2}\mathcal{M}\mathrm{on}}
\def\MonCat{\mathcal{M}\mathrm{on}\mathcal{C}\mathrm{at}}
\def\scrZ{\mathscr{Z}}
\def\Kan{\mathcal{K}\mathrm{an}}

\def\sAb{s\mathcal{A}\mathrm{b}}
\DeclareMathOperator{\sing}{sing}

\begin{document}
\renewcommand{\theenumi}{(\roman{enumi})}

\title[Refinements of the Chern-Dold Character]{Refinements of the Chern-Dold Character: Cocycle Additions in Differential Cohomology}
\author{Markus Upmeier}
\date{\today}
\maketitle

\begin{abstract}
The Chern-Dold character of a cohomology theory $E$ is a canonical transformation $E\rightarrow HV$ to ordinary cohomology. A spectrum representing $E$ gives homotopy theoretic cocycles for $E$, while $HV$ can be represented by singular cocycles. We construct a refinement of the Chern-Dold character to a transformation of the cocycle categories that takes the homotopical composition to the addition of singular cocycles. This is applied to construct additive structures at the level of differential cocycles for generalized differential cohomology.
\end{abstract}

\section{Introduction}

The Chern-Dold character (see~\cite{Dold-Relations-Cohomology}) is a natural transformation from an arbitrary generalized cohomology theory $E$ to ordinary cohomology with coefficients in the graded coefficient vector space $V^*=E^*(S^0)\otimes_\Z \R$:
\begin{equation}\label{classical-chern}
	\ch\colon E^*(X) \longrightarrow H^n(X;V)=\prod_{i+j=n} H^i(X; V^j)\qquad (X\in \Top^*)
\end{equation}
(cohomology, cochain groups, etc.~for pointed spaces are always reduced.)
Hence for ${E=H\Z}$ we get the standard map ${H^n(X,\Z)\rightarrow H^n(X,\R)}$. For topological $K$-theory, we get the ordinary Chern character $K^{0/1}(X) \rightarrow H^\mathrm{ev/od}(X)$.
\medskip

Ordinary cohomology can be represented by the abelian group of singular cocycles ${Z^n(X;V)}$. These form the objects of a strict monoidal category of cocycles $\scrZ^n(X)$. Similarly, given a spectrum ${(E_n, \varepsilon_n)}$ representing $E$ (with homeomorphisms ${\varepsilon_n^\mathrm{adj}\colon E_n\rightarrow \Omega E_{n+1}}$, which may always be arranged), maps ${X\rightarrow E_n}$ give cocycles for generalized cohomology. Loop composition in either direction gives two binary operations, identifying $E_n$ with ${\Omega^2 E_{n+2}}$. Endowed with these, the cocycles $\Map(X,E_n)$ for generalized cohomology form a $2$-monoidal category (see Section~\ref{2-monoidal}), a more sophisticated algebraic object than an abelian group.\medskip

We shall construct a refined Chern-Dold character between the cocycle categories in such a way that it \emph{preserves the algebraic structure} (strict addition of singular cocycles, loop composition):

\begin{theorem}\label{thm-refined-chern}
For any generalized cohomology theory $E$ and representing spectrum $(E_n,\varepsilon_n)$, there exists a natural family of $2$-monoidal functors
\begin{equation}\label{eq:chern-char}
	\ch_X\colon\Map(X, E_n) \rightarrow \scrZ^n(X).
\end{equation}
On isomorphism classes of objects, the functors \eqref{eq:chern-char} reduce to \eqref{classical-chern}.
\end{theorem}

The notation is established in Section~\ref{2-monoidal} where we also review the theory of $2$-monoidal categories. Theorem~\ref{thm-refined-chern} is proven in Section~\ref{ssec:proof-ref-chern} after having explained in Section~\ref{sec:cocycle-spectrum} that the construction of \eqref{eq:chern-char} passes through an intermediate step which mediates between the algebraic and homotopical point of view.\smallskip

One motivation is the following application (Section~\ref{sec:application}): recall from \cite{Hopkins-Singer} that a differential $n$-cocycle on a manifold $M$ consists of a continuous map ${c\colon M\rightarrow E_n}$, a differential form $\omega \in \Omega^n(M;V)$, and a cochain $h\in C^{n-1}(M;V)$ satisfying 
\[
	\delta h = \omega - c^*\iota_n.
\]
(here, $\iota_n \in Z^n(E_n; V)$ denote fundamental cocycles, see Section~\ref{ssec:fund-cocycles}.) A differential cocycle on $M\times [0,1]$ is regarded as an equivalence between the two cocycles on the boundary. The Hopkins-Singer differential cohomology group $\hat{E}^n(M)$ is by definition the set of equivalences classes of differential $n$-cocycles on $M$.\smallskip

Differential cocycles can also be organized into a category $\hat{\scrE}^n(M)$. A consequence of Theorem~\ref{thm-refined-chern} is that these may be given $2$-monoidal structures. Our main application is the following (nearly equivalent) statement:

\begin{theorem}\label{mainappladd}
For every choice of fundamental cocycles there exist reduced cochains $A_n \in {C^{n-1}(E_n\times E_n; V)}$ satisfying coherence relations \textup(see Section~\textup{\ref{sec:application}}\textup) so that
\begin{equation}\label{add-dif-cocycles}
	(c_1, \omega_1, h_1) + (c_2, \omega_2, h_2) = (\alpha_n(c_1,c_2),\omega_1+\omega_2, h_1+h_2+(c_1,c_2)^*A_n)
\end{equation}
gives an abelian group structure on the Hopkins-Singer differential extension $\hat{E}$.
\end{theorem}

In many cases it is important to have control of the algebraic structure at the level of differential cocycles. This is in sharp contrast to \cite{Hopkins-Singer}, where it is proven that the cohomology groups $\hat{E}^n(M)$ possess some abstract abelian group structure (they are identified as the homotopy groups of a spectrum whose structure maps are only abstractly chosen by a cofibrant replacement in a diagram model category, i.e.,~a choice of functorial sections): there is then no way of deciding which differential cocycle represents the sum.

\section{Theory of $2$-Monoidal Categories}\label{2-monoidal}
\noindent
The following is a special case of \cite[Section~5]{joyal1993braided} (or \cite[6.1]{MR2724388}) where the units coincide:% (a `multiplication' is nowadays known as a $2$-monoidal category):

\begin{definition}
A \emph{$2$-monoidal category} $(\scrC, \varobar, \varominus, I, \zeta)$ is a category $\scrC$ with two monoidal structures $\varobar, \varominus$ sharing a unit $I$ and a natural `interchange' isomorphism
\begin{equation}\label{interchange}
	\zeta_{A, B, C, D}\colon (A\varominus B)\varobar (C\varominus D) \rightarrow (A\varobar C)\varominus (B \varobar D).
\end{equation}
We require $I \varobar I = I = I\varominus I$ which along with \eqref{interchange} shall endow both
\begin{equation}\label{compatiblity-zeta}
	\varobar\colon (\scrC \times \scrC, \varominus \times \varominus) \rightarrow (\scrC, \varominus),\quad
	\varominus\colon (\scrC \times \scrC, \varobar \times \varobar) \rightarrow (\scrC, \varobar)
\end{equation}
with the structure of monoidal functors.\smallskip

A \emph{$2$-monoidal functor} $F\colon \scrC \rightarrow \scrD$ has monoidal structures ${F^\varobar: (\scrC, \varobar) \rightarrow (\scrD, \varobar)}$, ${F^\varominus\colon (\scrC, \varominus) \rightarrow (\scrD, \varominus)}$ whose unit constraints ${F^{\varobar, \varominus}(I)=I}$ are the identity.
We require commutative diagrams for all objects $A,B,C,D$
\begin{equation}\label{axiom-2-mon-functor}
\begin{tikzpicture}[baseline=(current  bounding  box.center)]
\matrix (m) [matrix of math nodes, column sep=1.5em, row sep=2em]
{
	F \left( (A \varominus B) \varobar (C\varominus D) \right)&
	F\left(  (A\varobar C) \varominus (B\varobar D) \right)\\
	F(A\varominus B)\varobar F(C\varominus D)&
	F(A\varobar C)\varominus F(B\varobar D)\\
	(FA\varominus FB)\varobar (FC\varominus FD)&
	(FA \varobar FC) \varominus (FB\varobar FD).\\
};
\path[->, font=\scriptsize] (m-1-1) edge node [above] {$F(\zeta^\scrC)$} (m-1-2);
\path[->, font=\scriptsize] (m-3-1) edge node [below] {$\zeta^\scrD$} (m-3-2);

\path[->, font=\scriptsize] (m-1-1) edge node [left] {$F^\varobar_{A\varominus B, C\varominus D}$} (m-2-1);
\path[->, font=\scriptsize] (m-2-1) edge node [left] {$F^\varominus_{A,B} \varobar F^\varominus_{C,D}$} (m-3-1);

\path[->, font=\scriptsize] (m-1-2) edge node [right] {$F^\varominus_{A\varobar C, B\varobar D}$} (m-2-2);
\path[->, font=\scriptsize] (m-2-2) edge node [right] {$F^\varobar_{A,C} \varominus F^\varobar_{B,D}$} (m-3-2);
\end{tikzpicture}
\end{equation}
\end{definition}

Restricting to small categories $\scrC$, these definitions give a category $\TwoMon$. We call $F$ an \emph{equivalence} if it is an equivalence of the underlying categories.\smallskip

We assume familiarity with monoidal categories as presented in \cite{leinster2004higher}.
%We recall that a monoidal functor $F\colon \scrC \rightarrow \scrD$ comes equipped with natural `structure' maps $F_{A,B}\colon F(A\otimes B) \rightarrow F(A)\otimes F(B)$ for every pair of objects $A,B$. These are required to satisfy coherence conditions, need to agree when monoidal functors are equal, and are taken into account when composing monoidal functors. On the other hand, for a transformation to be monoidal is a property, not additional data.
Equality of $2$-monoidal functors $F=G$ means that both monoidal structures ${F^\varobar = G^\varobar}$, ${F^\varominus = G^\varominus}$ agree.
$F$ is called \emph{strict} if both $F^\varobar, F^\varominus$ are strict. Restricting to such functors gives the subcategory $\TwoMon_\mathrm{strict}$.
Denote by $\Cat$ ($\MonCat$) the category of small (monoidal) categories.\medskip

We think of a functor $\scrC\colon \scrI \rightarrow \Cat$ as a \emph{\textup(natural\textup) family of categories}. By a \emph{$2$-monoidal structure} on $\scrC$ we mean a lift to a functor $\hat\scrC\colon{\scrI \rightarrow \TwoMon}$: for every $X\in \scrI$ we have $2$-monoidal categories $\scrC_X$ and every morphism $f\in \scrI(X,Y)$ gives a $2$-monoidal functor $\scrC(f)\colon \scrC_X \rightarrow \scrC_Y$. A natural transformation ${F\colon \scrC\Rightarrow \scrD}$ between two families $\scrC,\scrD\colon {\scrI\rightarrow\TwoMon}$ may be viewed as a \emph{\textup(natural\textup) family of $2$-monoidal functors}: for every $X\in \scrI$ we have a $2$-monoidal functor $F_X\colon \scrC_X \rightarrow \scrD_X$.
Naturality means that for $f\colon X\rightarrow Y$ in $\scrI$ we get a commutative diagram in $\TwoMon$:
\begin{equation}\label{eqn:naturality-gen}
\begin{aligned}
\xymatrix{
\scrC_X\ar[d]_{\scrC(f)}\ar[r]^{F_X}&\scrD_X\ar[d]^{\scrD(f)}\\
\scrC_Y\ar[r]_{F_Y}&\scrD_Y
}
\end{aligned}
\end{equation}
%We call $F$ a family of (strict) $2$-monoidal isomorphisms (equivalences) if this is true for every component $F_X$.

The motivating example for the theory of $2$-monoidal categories is the following:

\begin{example}\label{example-1}
The fundamental groupoid $\Pi_1 \Omega^2 X$ of the double loop space of a space $X\in \Top^*$ has monoidal structures by vertical and horizontal composition:
\begin{align*}
	f\varobar g(s,t) &= \begin{cases}
		f(2s,t),&s\leq 1/2,\\
		g(2s-1,t),&s\geq 1/2,
	\end{cases}
	&f\varominus g(s,t) &= \begin{cases}
		f(s,2t),&t\leq 1/2,\\
		g(s,2t-1),&t\geq 1/2.
	\end{cases}
\end{align*}
%The associativity and unit constraints are given by the standard homotopies that are used to show that the fundamental group is indeed a group --- for formulas see \cite[4.1.1]{dr-arbeit}.
The interchange $\zeta$ is the identity and the common unit $I$ is the base-point map. Since maps $X\rightarrow Y$ preserve $\varobar, \varominus, I$ we get a functor ${\Pi_1\Omega^2\colon \Top^* \rightarrow \TwoMon_\mathrm{strict}}$.
\end{example}

There is a \emph{unique} way to transport a $2$-monoidal structure along an isomorphism $F\colon\scrC\rightarrow \scrD$ of categories, turning $F$ into a strict $2$-monoidal functor (so $X\varobar Y = F^{-1}(FX\varobar FY)$ and $F\zeta^\scrC = \zeta^\scrD$).
As usual, uniqueness implies functoriality: any functor ${\scrC\colon\scrI \rightarrow \Cat}$ naturally isomorphic to ${\scrD\colon\scrI \rightarrow \TwoMon}$ may be uniquely lifted along the forgetful functor to ${\hat\scrC\colon \scrI\rightarrow \TwoMon}$, making every component ${\hat\scrC_X \rightarrow \scrD_X}$ of the transformation a strict $2$-monoidal functor.

\begin{definition}\label{def:explanation}
For $X\in\Top^*$, let $\Map(X,E_n) = \Pi_1 E_n^X$ be the fundamental groupoid of the pointed mapping space (pointed maps ${X\rightarrow E_n}$ and homotopies).
The structure maps induce a natural family of isomorphisms $\Map(X,E_n) \cong \Pi_1\Omega^2 E_{n+2}^X$ to $2$-monoidal categories from Example~\ref{example-1}. Transporting, we get
\[
	\Map(-,E_n)\colon \Top^* \rightarrow \TwoMon.
\]
(Recall that $(E_n, \varepsilon_n)$ is a spectrum representing the cohomology theory $E$.)
\end{definition}

%We now give some examples of natural families of $2$-monoidal categories:

\begin{examples}\label{examples-2}
\noindent
\textbf{(i)} Let $A$ be a topological or simplicial abelian group. On $\Pi_1 A$ we take $\varobar=\varominus=+$, $I=0$, and the identity as the interchange. Since $+$ and $0$ are preserved by group homomorphisms, we get a functor $\Pi_1\colon \sAb \rightarrow \TwoMon_\mathrm{strict}$.\vskip1ex

\noindent
\textbf{(ii)} Any monoidal category $(\scrC, \otimes)$ may be regarded as being $2$-monoidal by taking $\varobar = \varominus = \otimes$ and $\zeta = \id$. This yields a functor $\MonCat \rightarrow \TwoMon$.\vskip1ex

\noindent
\textbf{(iii)} For every cochain complex $(C^*, d)$ and $n\in \Z$, define a strict monoidal category $\scrZ^n$ of $n$-cocycles: objects are $x \in C^n$ with $dx=0$. A morphism $x\rightarrow y$ consists of an $\im(d)$-coset of elements $u\in C^{n-1}$ with $du = x - y$. Composition and the monoidal structure are both given by addition. Combined with (ii) we obtain for each $n\in \Z$ a functor $\scrZ^n\colon\Ch \rightarrow \TwoMon_\mathrm{strict}$ on cochain complexes.
\end{examples}

Recall ${C^*(X;V) = \prod_{i+j=*} C^i(X;V^j)}$ for a graded vector space $V^*$. Being fixed in our discussion as ${V=E^*(S^0)\otimes_\Z \R}$, we will simply write $C^*(X)=C^*(X;V)$. Similarly, we shall write $Z^*(X) = Z^*(X;V)$ for the singular cocycles of $X$ with coefficients in $V$. For chains groups $C_*(X)$, this convention is \emph{not} adopted.

\begin{definition}
The reduced cochain complex $C^*(-;V)$ from $\Top^*$ to $\Ch$ composed with (iii) gives the functor $\scrZ^n\colon \Top^* \rightarrow \TwoMon_\mathrm{strict}$.
\end{definition}

The proofs of the following two propositions are given in Appendix \ref{appendix}. 

\begin{proposition}\label{prop:equivalent-functors}
Let $\scrC, \scrD, \scrE$ be $2$-monoidal categories and suppose $F\colon \scrC \rightarrow \scrD$, $G\colon \scrD\rightarrow \scrE$ are \textup(ordinary\textup) functors of the underlying categories. Let $H=G\circ F$.
\begin{enumerate}
\item
If $F$ is an equivalence and $F, H$ have $2$-monoidal structures, then $G$ has a unique $2$-monoidal structure such that $H=G\circ F$ as $2$-monoidal functors.
\item
If $G$ is an equivalence and $G, H$ have $2$-monoidal structures, then $F$ has a unique $2$-monoidal structure such that $H=G\circ F$ as $2$-monoidal functors.
\end{enumerate}
\end{proposition}

\begin{proposition}\label{prop:equivalent-stuff}
Let ${\scrC\colon \scrI \rightarrow \Cat}$, ${\scrD\colon \scrI \rightarrow \TwoMon}$ be functors. Suppose $F\colon \scrC \Rightarrow \scrD$ is a nat.~transformation of $\Cat$-valued functors whose components are equivalences
\[
	F_X\colon \scrC_X \xrightarrow{\sim} \scrD_X,\quad X\in\scrI.
\]
Then we may lift $\scrC$ to $\hat{\scrC}\colon\scrI\rightarrow\TwoMon$ and promote $F$ to a natural family of $2$-monoidal equivalences $\hat{F}_X\colon \hat{\scrC}_X \rightarrow \scrD_X$ \textup(i.e.,~a natural transformation $F\colon \hat{\scrC} \Rightarrow \scrD$\textup). % \textup(natural $\scrD(i\rightarrow j)\circ \hat{F}_i = \hat{F}_j\circ\hat{\scrC}(i\rightarrow j)$ as $2$-monoidal functors\textup).
\end{proposition}

Of course, this also holds in the dual situation $\scrC\colon \scrI \rightarrow \TwoMon$, $\scrD\colon \scrI\rightarrow \Cat$. We emphasize that the lift $\hat{\scrC}$ is \emph{not unique}, but can still be \emph{chosen} functorially.

%
%\begin{proposition}\label{prop:equivalent-stuff}
%Let ${F\colon \scrC \rightarrow \Cat}$, ${G\colon \scrC \rightarrow \TwoMon}$ be functors. Suppose $u$ is a natural transformation $F\rightarrow G$ of $\Cat$-valued functors consisting of equivalences
%\[
%	u_C\colon F(C) \xrightarrow{\sim} G(C),\qquad C\in\scrC.
%\]
%Then we may lift $F$ to ${\hat{F}\colon\scrC\rightarrow \TwoMon}$ and promote $u$ to a natural family of equivalences $u_C\colon \hat{F}(C) \simeq G(C)$ of $2$-monoidal categories.
%\end{proposition}

\begin{example}\label{ex:groupoid-two-kan}
On the category $\Kan^*$ of pointed Kan complexes consider
\[
	\scrC\colon\Kan^* \xlongrightarrow{\Pi_1 \Omega^2} \Cat,\quad
	\scrD\colon \Kan^* \xlongrightarrow{|\cdot|} \Top^*
	\xlongrightarrow{\Pi_1\Omega^2} \TwoMon.
\]
($F_X$ is induced by geometric realization of points and paths in $X$.) Hence the fundamental groupoids $\Pi_1 \Omega^2 K$ for pointed Kan complexes (see~\cite{Goerss-Jardine}) can be given \emph{functorial} $2$-monoidal structures $\Pi_1\Omega^2\colon \Kan^* \rightarrow \TwoMon$.
\end{example}

One of the main results of \cite[Section~5]{joyal1993braided} (or see \cite[Proposition~6.11]{MR2724388}) is that there is an equivalence from $\TwoMon$ to braided monoidal categories:

\begin{theorem}\label{prop:joyal-street}
From a $2$-monoidal structure $\scrC$ one can construct braidings on $(\scrC, \varobar)$ and $(\scrC,\varominus)$. The identity functor may be viewed as a braided monoidal functor ${e\colon(\scrC, \varobar) \rightarrow (\scrC, \varominus)}$ with unit constraint $e_I = \id$ and structure maps
% For example, the braid on $\varobar$ is given by the familiar composition
%\begin{align*}
%	A\varobar B &\xleftarrow{\lambda^\varominus \varobar \rho^\varominus} (1\varominus A)\varobar (B\varominus 1)\cong(1\varobar B)\varominus (A\varobar 1) \xrightarrow{\lambda^\varobar \varominus \rho^\varobar} B\varominus A\\
%	&\xleftarrow{\rho^\varobar\varominus \lambda^\varobar} (B\varobar 1)\varominus (1\varobar A)\cong(B\varominus 1)\varobar (1\varominus A) \xrightarrow{\lambda^\varominus \varobar \rho^\varominus} B\varobar A.
%\end{align*}
\begin{equation}\label{monstruct-mutdistrib}
	e_{A,B}^{\varobar, \varominus}: A\varobar B \xleftarrow{\;\rho^\varominus\varobar \lambda^\varominus} (A\varominus I)\varobar (I\varominus B) \cong (A\varobar I)\varominus (I\varobar B) \xrightarrow {\rho^\varobar\varominus\lambda^\varobar} A\varominus B.
\end{equation}
\textup(here $\lambda^\varobar, \rho^\varobar$ are the unit constraints on $(\scrC,\varobar)$ and similarly for $\varominus$.\textup)
\end{theorem}

The double loop space $\Omega^2 A$ of a topological abelian group $A$ (base-point $0$) is again an abelian group. Example~\ref{example-1} and Example~\ref{examples-2}~(i) give two different ways of viewing the fundamental groupoid $\Pi_1 \Omega^2 A$ as a $2$-monoidal category.

\begin{lemma}
For every topological abelian group $A$, the identity functor may be endowed canonically with the structure of a $2$-monoidal functor:
\begin{equation}\label{plus-varobar-varominus}
	(\Pi_1 \Omega^2 A, \varobar, \varominus,\const_0, \id) \longrightarrow (\Pi_1 \Omega^2 A, +, +,0, \id).
\end{equation}
\end{lemma}
\begin{proof}
Since the operations $\varobar$ and $+$ are mutually distributive, $(\Pi_1\Omega^2 A, \varobar, +, \id)$ defines a $2$-monoidal category and Theorem~\ref{prop:joyal-street} gives a canonical monoidal structure $e_{f,g}^{\varobar, +}$ on the identity functor. Similarly, we get a monoidal structure $e_{f,g}^{\varominus,+}$. It remains to show the commutativity of \eqref{axiom-2-mon-functor}. Suppose $\gamma, \phi \colon [0,1]\rightarrow [0,1]^2$ satisfy $\gamma\leq \phi$ componentwise. For $f \in \Pi_1\Omega^2 A$ define a homotopy that places $f$ into the rectangles bounded by $\gamma$ and $\phi$. Viewing $f$ as a map ${[0,1]^2 \rightarrow A}$ taking the boundary to zero and extended to the plane by zero, we may write
\[
	\{\gamma,\phi\}_f(t,x,y) = f\left(  \frac{x-\gamma_1(t)}{\phi_1(t)-\gamma_1(t)}, \frac{y-\gamma_2(t)}{\phi_2(t)-\gamma_2(t)} \right).
\]
A path homotopy $\gamma^s\leq \phi^s$ (parameter $s$) gives a homotopy $\{\gamma^s, \phi^s\}_f$ of  homotopies.

For paths $u, v\colon [0,1]\rightarrow X$ with $u(1)=v(0)$ let $u\star v$ denote `$u$ followed by $v$'.
Write $\alpha(t)=(1+t)/2$, $\beta(t)=(1-t)/2$ and $c(t)=c$ for fixed $c\in [0,1]$. Equation \eqref{monstruct-mutdistrib} defines $e_{f,g}^{\varobar,+}$ as ${\left\{ (0,0),(\alpha,1) \right\}_f + \left\{ (\beta,0), (1,1) \right\}_g}$. Similarly, $e_{f,g}^{\varominus,+}= {\left\{ (0,0), (1,\alpha)\right\}_f + \left\{(0,\beta),(1,1)\right\}_g}$. Performing the composition,
\begingroup\setlength{\medmuskip}{1mu}
\begin{align*}
{(e_{f,g}^{\varobar,+} + e_{h,j}^{\varobar,+}) e_{f\varominus g, h\varominus j}^{\varobar,+}}=
&\big\{
(0,0),
(\alpha,1/2)\star(1,\alpha)
\big\}_f
+
\big\{
(0,1/2)\star(0,\beta),
(\alpha,1)\star(1,1)
\big\}_g\\
+&
\big\{
(\beta,0)\star(0,0),
(1,1/2)\star(1,\alpha)
\big\}_h
+
\big\{
(\beta,1/2)\star(0,\beta),
(1,1)
\big\}_j,\\
{(e_{f,h}^{\varobar,+} + e_{g,j}^{\varobar,+}) e_{f\varobar h, g\varobar j}^{\varominus,+}}=&\big\{
(0,0),
(1/2,\alpha)\star (\alpha,1)
\big\}_f
+
\big\{
(0,\beta)\star (0,0),
(1/2,1)\star (\alpha,1)
\big\}_g\\
+&
\big\{
(1/2,0)\star(\beta,0),
(1,\alpha)\star (1,1)
\big\}_h
+
\big\{
(1/2,\beta)\star(\beta,0),
(1,1)
\big\}_j.
\end{align*}
\endgroup
These are homotopic, since any two paths in $[0,1]^2$ are homotopic by a linear homotopy, so $\{\gamma,\phi\}_f \simeq \{\gamma',\phi'\}_f$ for any $\gamma\leq\phi,\gamma'\leq\phi'$ and any $f$.
%\[
%	\left(e_{f,g}^{\varobar,+}\right)_t(x,y) = f\left(\frac{2}{1+t}x,y\right) + g\left(\frac{2}{1+t}x-\frac{1-t}{1+t},y\right).
%\]
\end{proof}

\section{The Cocycle Spectrum of a Space}\label{sec:cocycle-spectrum}

In this section, we shall construct an auxiliary object which mediates between the algebraic and homotopical point of view. We assume familiarity with simplicial sets (see \cite{Goerss-Jardine}). Recall that the Moore complex $C(K)_*$ of a pointed simplicial set $K$ has the group $\Z K_n / \Z\pt$ as $n$-chains. We adopt the standard notation ${C(X)=C(\sing X)}$ for $X\in \Top^*$.  Let $L_+$ denote $L$ with a disjoint base-point.

\begin{definition}
The $n$-th space of the \emph{cocycle spectrum} is the simplicial vector space of chain maps ($V[-n]_* =V^{n-*}$ for $*\geq 0$ with zero differential is viewed as an object of the category $\Ch_{\geq 0}$ of non-negative chain complexes):
\begin{equation}\label{eq:cocycle-spectrum}
	Z^n(K\wedge \Delta^\bullet_+) = \Ch_{\geq 0}\big(	C(K\wedge\Delta^\bullet_+)_*, V[-n]_*	\big)
	=\prod_{i+j=n} Z^i(K\wedge \Delta^\bullet_+; V^j)
\end{equation}
%(here $\Delta^\bullet\colon \Delta \rightarrow \sSet,\; [n]\mapsto \Delta(-,[n])$ is the standard cosimplicial space)
Being fixed, we omit $V$ from the notation on the left of \eqref{eq:cocycle-spectrum}.
\end{definition}

The spaces $Z^n(K\wedge \Delta^\bullet_+)$ are mapping spaces $\mathrm{Map}_{\mathrm{Ch}}(C(K), V[-n])$ in the $\infty$-category of non-negative chain complexes \cite[Section~13]{Lurie-Derived-I}, so the cocycle spectrum may be regarded as a function spectrum construction. Weakly equivalent spaces were introduced in \cite{Hopkins-Singer}, but we will see below that it is crucial to work with \eqref{eq:cocycle-spectrum}.\bigskip

Recall the \emph{Alexander-Whitney} and \emph{Eilenberg-Zilber} chain maps
\[
	EZ\colon C(K)\otimes C(L)\rightarrow C(K\wedge L),\quad
	AW\colon C(K\wedge L)\rightarrow C(K)\otimes C(L). 
\]
The \emph{slant product} of a cochain $u$ with a chain $e$ is the cochain ${u/e}$ defined by
${(u/e)(d) = u\left( EZ(d\otimes e) \right)}$. Since $EZ$ is a chain map, we get a Stokes formula
\begin{equation}\label{stokes-slant}
	(\delta u/e)  = \delta(u/e) - (-1)^{|u|+|e|} u/\partial e.
\end{equation}
Let $[\Delta^i_+] \in C_i(\Delta^i_+)$ and $[ S^1] \in C_1(S^1)$ denote the canonical chains ($S^1=\Delta^1/\partial \Delta^1$).\medskip

\noindent
We take from \cite[D.13]{Hopkins-Singer} the isomorphism `slant product along the $i$-chain $[\Delta^i_+]$'
\begin{equation}\label{iso-tohomotopy}
	\pi_i \left( Z^n(K\wedge\Delta^\bullet_+), 0 \right) \cong H^{n-i}(K; V),\quad
	f\in Z^n(K\wedge \Delta^i_+) \mapsto f / [\Delta^i_+].
\end{equation}
(this fact is also proven in \cite[Lemma~5.7]{dr-arbeit}.)

\begin{lemma}\label{lem:coequalizer-proof}
There is a canonical isomorphism of simplicial sets
\begin{equation}\label{eqn:loop-inner-circle}
\Omega Z^n(K\wedge \Delta^\bullet_+) \cong Z^n(K\wedge\Delta^\bullet_+\wedge S^1).
\end{equation}
\end{lemma}
\begin{proof}
The usual subdivision of the prism ${h_i\colon \Delta^{k+1}\rightarrow \Delta^k\times \Delta^1}$ for $i=0,\ldots,k$ \cite[p.17]{Goerss-Jardine} leads to a coequalizer diagram in pointed simplicial sets $\sSet^*$:
\[\xymatrix@R=0.5cm@C=1.5cm{
K\wedge \Delta^k_+\ar[r]^{\id\wedge d^{j+1}}\ar[d]_{\mathrm{in}_j}						&	K\wedge \Delta^{k+1}_+\ar[d]^{\mathrm{in}_j}\\
\bigvee_{j=-1}^k K\wedge\Delta^k_+	\ar@<+3pt>[r]\ar@<-3pt>[r]		&	\bigvee_{j=0}^k K\wedge \Delta^{k+1}_+\ar[r]^-{h_0\vee \ldots \vee h_k}	&	K\wedge \Delta^k_+\wedge S^1\\
K\wedge \Delta^k_+\ar[r]_{\id\wedge d^{j+1}}\ar[u]^{\mathrm{in}_j}						&	K\wedge \Delta^{k+1}_+\ar[u]_{\mathrm{in}_{j+1}}
}\]
(letting $\mathrm{in}_l$ be the constant base-point maps if $l=-1, k+1$.) The reduced Moore complex $C\colon \sSet^*\rightarrow \Ch_{\geq 0}$ is a left-adjoint and therefore preserves colimits. Hence a $k$-simplex $f\in Z^n(K\wedge \Delta^k_+\wedge S^1)$ is a chain map defined on the coequalizer of
\[\xymatrix@1{
\bigoplus_{j=-1}^k C(K\wedge \Delta^k_+)\ar@<-2pt>[r]\ar@<+2pt>[r] & \bigoplus_{j=0}^k C(K\wedge \Delta^{k+1}_+).
}
\]
This amounts to a sequence of maps $f_i \in Z^n(K\wedge \Delta^{k+1}_+)$ which are compatible exactly so as to represent a $k$-simplex of the loop space $\Omega Z^n(K\wedge \Delta^\bullet_+)$ (a $k$-simplex of a simplicial loop space $\Omega L$ may be described as a sequence of ${(k+1)}$-simplices ${f_0, \ldots, f_k}$ with ${d_i f_i = d_i f_{i-1}}$ and ${d_0 f_0 = d_{k+1} f_k = *}$).
\end{proof}

\begin{definition}
Letting `$\incl$' be given by the canonical $1$-chain $[S^1]$, consider
\begin{equation}\label{eqn:diag-of-quasi-isos}
\xymatrix@C=1cm{
C(K\wedge \Delta^\bullet_+)\otimes \Z[1]\ar[r]^-{\id\otimes\incl}&
C(K\wedge \Delta^\bullet_+)\otimes C(S^1)\ar[r]^-{EZ}&
C(K\wedge \Delta^\bullet_+\wedge S	^1).
}
\end{equation}
Combining that $-\otimes \Z[1]$ is the shift $[-1]$ with Lemma~\ref{lem:coequalizer-proof}, pullback along \eqref{eqn:diag-of-quasi-isos} gives the \emph{costructure maps} (`co' because they map away from the loop space)
\begin{equation}\label{costructure-maps}
	\psi\colon \Omega Z^n(K\wedge \Delta^\bullet_+)=Z^n(K\wedge\Delta^\bullet_+\wedge S^1)\rightarrow Z^{n-1}(K\wedge\Delta^\bullet_+).
\end{equation}
\end{definition}

\begin{proposition}
The costructure maps $\psi$ are natural weak equivalences.
%Hence for each pointed $K$, the groups $Z^n(K\wedge \Delta^\bullet_+)$, $n\in \Z$, form a `weak cospectrum'.
\end{proposition}
\begin{proof}
This follow from the standard fact that the suspension may be expressed as the slant product along $S^1$: we show that we have commutative diagrams
\[\xymatrix@C=1.3cm{
\pi_k Z^n(K\wedge \Delta^\bullet_+\wedge S^1)\ar[d]_{\eqref{iso-tohomotopy}}^\cong\ar[r]^{\pi_k(\psi)}&
\pi_k Z^{n-1}(K\wedge \Delta^\bullet_+)\ar[d]^{\eqref{iso-tohomotopy}}_\cong\\
H^{n-k}(K\wedge S^1)\ar[r]_{\mathrm{susp}}^\cong&H^{n-k-1}(K).
}\]
Explicitly, for $f\in \pi_k Z^n(K\wedge \Delta^\bullet_+\wedge S^1)$ we need to compare the two assignments on chains $\sigma \in C_{n-k-1}(K)$ given by
\begin{equation}\label{EZcohom}
	f\circ EZ\big ( \Delta^k \otimes EZ(\sigma\otimes S^1) \big),\quad
	f\circ EZ\big(   EZ(\Delta^k\otimes \sigma) \otimes S^1  \big).
\end{equation}
Since $EZ$ is coassociative up to chain homotopy, we have a homomorphism $h$ so that the difference is (using that $f$ is a chain map and $V_*$ has zero differential)
\[
	f\circ (\partial h\sigma + h\partial\sigma) = \partial f (h\sigma) + fh\partial \sigma = 0+\delta(fh)\sigma
\]
Therefore, both elements in \eqref{EZcohom} represent the same cohomology class.
\end{proof}

\begin{definition}
By the costructure maps on $\sing (E_n^X)$ we mean the isomorphisms
\begin{equation}\label{fspec-struct}
	\Omega \sing(E_{n+1}^X) \cong \sing(\Omega E_{n+1}^X) \xrightarrow{\sing(\varepsilon_n^\mathrm{adj})^{-1}_*} \sing(E_n^X).
\end{equation}
\end{definition}

\section{The $2$-Monoidal Chern-Dold Transformation}\label{sec:refined-chern}

Our construction of \eqref{eq:chern-char} will factor into three $2$-monoidal functors
\begin{equation}\label{eq:three-steps}
	\ch_X\colon \Map(X, E_n) 	\xrightarrow{\alpha} \scrZ^n_{\varobar,\varominus}(X)
					\xrightarrow{\beta} \scrZ^n_+(X)
					\xrightarrow{\gamma} \scrZ^n(X).
\end{equation}
We begin by explaning the new categories in \eqref{eq:three-steps}. By Example~\ref{examples-2}~(i), addition gives a strict $2$-monoidal structure on the fundamental groupoid $\Pi_1{Z^n(\sing X\wedge \Delta^\bullet_+)}$ that we denote by $\scrZ^n_+(X)$.
Hence the objects of $\scrZ^n_+(X)$ are singular cocycles $Z^n(X)$ while the morphisms $h\colon d_1 h \rightarrow d_0 h$ are cocycles $h\in Z^n(\sing X\wedge \Delta^1_+)$. Another way to get a $2$-monoidal structure on the same category is to note that the costructure map $\psi$ induces equivalences of categories% (natural in $X$)
\begin{equation}\label{double-cocycle}
	\Pi_1 \Omega^2 Z^{n-2}(\sing X \wedge \Delta^\bullet_+) \xrightarrow{\sim} \Pi_1 Z^n(\sing X\wedge \Delta^\bullet_+).
\end{equation}
The left-hand side has a natural $2$-monoidal structure by Example~\ref{ex:groupoid-two-kan}. According to Proposition~\ref{prop:equivalent-stuff}, we may \emph{choose} natural $2$-monoidal structures $\scrZ^n_{\varobar,\varominus}(X)$ on the right-hand categories, making \eqref{double-cocycle} a natural $2$-monoidal equivalence.\medbreak
%
%\newcounter{temptheorem}
%\setcounter{temptheorem}{\thetheorem}
%\setcounter{theorem}{0}
%\begin{theorem}
%For any generalized cohomology theory $E$, there exists a natural family of $2$-monoidal functors \textup(i.e.,~a natural transformation $\Map(-,E_n)\Rightarrow \scrZ^n(-)$\textup)
%\begin{equation}\label{eq:chern-char}
%	\ch_X\colon\Map(X, E_n) \rightarrow \scrZ^n(X),
%\end{equation}
%where $(E_n, \varepsilon_n)$ is a spectrum representing $E$. On isomorphism classes of objects, the functors \eqref{eq:chern-char} reduce to the classical Chern-Dold character \eqref{classical-chern}.\end{theorem}
%\setcounter{theorem}{\thetemptheorem}

\subsection{Fundamental Cocycles}\label{ssec:fund-cocycles}

Recall that \emph{fundamental cocycles} are a family of singular cocycles $\iota_n \in Z^n(E_n; V)$ implementing the Chern-Dold character via
\[
\ch(f) = f^*[\iota_n],\qquad \forall f \in E^n(X) = [X,E_n].
\]
By \cite[4.8]{Hopkins-Singer}, there is a choice satisfying $\varepsilon_{n}^*\iota_{n+1}/[S^1] = \iota_n$, where $\varepsilon_{n}\colon {E_n\wedge S^1\rightarrow E_{n+1}}$ are the structure maps (a more detailed proof of this assertion may be found in \cite[Section~3.1.2]{dr-arbeit}).
Stated differently, we have chain maps
\[
	\iota_n\colon C(E_n)=C(\sing E_n) \rightarrow V[-n]
\]
fitting into commutative diagrams
\begin{equation}\begin{aligned}\label{fund-cocycle-comp}
\xymatrix@R=0.5cm@C=1.75cm{
	C(E_n)\otimes \Z[1]\ar[r]^{EZ(\id\otimes\incl)}\ar@{=}[d]&
	C(E_n \wedge S^1)\ar[r]^{C(\varepsilon_n)}&
	C(E_{n+1})\ar[d]^{\iota_{n+1}}\\
	C(E_n)[-1]\ar[rr]_{\iota_n[-1]}&&V[-n-1].
}\end{aligned}\end{equation}

\begin{definition}
We define simplicial maps $A_n\colon\sing(E_n^X) \rightarrow Z^n(\sing X \wedge \Delta^\bullet_+)$ by
\[
	A_n(f)\colon C(\sing X \wedge \Delta^k_+) \xrightarrow{C(f^\adj)} C(\sing E_n) \xrightarrow{\iota_n} V[-n],
\]
where, for a $k$-simplex $f\colon X\wedge |\Delta^k_+| \rightarrow E_n$ of $\sing(E_n^X)$, we use the unit to write
\[
f^\adj\colon \sing X\wedge \Delta^k_+ \rightarrow \sing X \wedge \sing |\Delta^k_+| = \sing(X\wedge |\Delta^k_+|) \xrightarrow{\sing f} \sing E_n.
\]

\end{definition}

\begin{lemma}\label{lem:com-costructure}
The maps $A_n$ commute with the costructure maps:
\[\xymatrix@C=1.25cm{
\sing(E_n^X)\ar[r]^-{A_n}&Z^n(\sing X \wedge \Delta^\bullet_+)\\
\Omega\sing(E_{n+1}^X)\ar[u]^{\eqref{fspec-struct}}\ar[r]_-{\Omega A_{n+1}}& \Omega Z^{n+1}(\sing X \wedge \Delta^\bullet_+).\ar[u]_{\eqref{costructure-maps}}
}\]
\end{lemma}
\begin{proof}
For $f: X\wedge |\Delta^\bullet_+| \wedge S^1 \rightarrow E_{n+1}$ let $g: X\wedge |\Delta^\bullet_+| \rightarrow E_n$ be the map $(\varepsilon_n^\mathrm{adj})^{-1}f$ with $\varepsilon_n\circ(g\wedge 1_{S^1}) = f$. If we write $K=\sing X$, the counit gives a simplicial map
\[
	\varphi(f)\colon K \wedge \Delta^\bullet_+\wedge S^1 \rightarrow \sing(X \wedge |\Delta^\bullet_+| \wedge S^1) \rightarrow \sing E_{n+1}.
\]
 Unwinding the definitions of \eqref{costructure-maps}, \eqref{fspec-struct}, and $A$, we see that we need to compare
\begin{align*}
	C(K\wedge \Delta^\bullet_+)\otimes \Z[1]&\xrightarrow{1\otimes\incl} C(K\wedge\Delta^\bullet_+)\otimes C(S^1)\xrightarrow{EZ} C(K\wedge\Delta^\bullet_+\wedge S^1)\\
	&\xlongrightarrow{\;\;\varphi(f)_*\;\;\,} C(E_{n+1}) \xrightarrow{\iota_{n+1}} V[-n-1]
\end{align*}
with the shift by one of
\[
	C(K\wedge \Delta^\bullet_+) \xlongrightarrow{\varphi(g)_*} C(E_n) \xlongrightarrow{\iota_n} V[-n].
\]
But these maps appear as the outer maps in the diagram
\[\xymatrix{
C(K\wedge\Delta^\bullet_+)\otimes \Z[1]\ar[r]^-{g_*\otimes 1}\ar[d]_{1\otimes\incl}&C(E_n)\otimes \Z[1]\ar[d]^{1\otimes\incl}\ar[r]^{\iota_n\otimes 1} & V[-n]\otimes\Z[1]\\
C(K\wedge\Delta^\bullet_+)\otimes C(S^1)\ar[r]^-{g_*\otimes 1}\ar[d]_{EZ}& C(E_n)\otimes C(S^1)\ar[d]^{EZ}\ar@{}[r]|-{\text{\normalsize\eqref{fund-cocycle-comp}}}&\\
 C(K\wedge\Delta^\bullet_+\wedge S^1)\ar@/_0.5cm/[rr]_{f_*}\ar[r]^-{(g\wedge 1)_*} & C(E_n\wedge S^1)\ar[r]^{(\varepsilon_n)_*}
& C(E_{n+1})\ar[uu]_{\iota_{n+1}}
}\]
which commutes by naturality of $EZ$ and the compatibility \eqref{fund-cocycle-comp}.
\end{proof}

\subsection{Proof of Theorem~\ref{thm-refined-chern}}\label{ssec:proof-ref-chern} The proof is divided into three steps:

\begin{lemma}\label{lem:alpha}
The maps $A_n$ induce a natural family of $2$-monoidal functors
\[
\Pi_1 A_n=\alpha_X\colon \Map(X, E_n) \rightarrow \scrZ^n_{\varobar,\varominus}(X).
\]
\end{lemma}
\begin{proof}
Lemma~\ref{lem:com-costructure} asserts that the diagram of ordinary categories underlying
\[\xymatrix@C=3cm{
	\Map(X,E_n)\ar@{-->}[r]^{\alpha_X}					&	\scrZ^n_{\varobar,\varominus}(X)\\
	\Pi_1 \Omega^2 E_{n-2}^X\ar[u]^\cong\ar[r]^{\Pi_1 \Omega^2 (A_{n-2})}	&	\Pi_1 \Omega^2 Z^{n-2}(\sing X\wedge \Delta^\bullet_+)\ar[u]\ar[u]_{\eqref{double-cocycle}}
}\]
commutes. On the bottom is the $2$-monoidal functor $\Pi_1 \Omega^2$ from Example~\ref{ex:groupoid-two-kan} and the vertical functors are $2$-monoidal by definition of the $2$-monoidal structure on the categories upstairs. Proposition~\ref{prop:equivalent-functors}~(i) states that there is a unique way to put a $2$-monoidal structure $\alpha_X$ on $\Pi_1 A$ so as make this diagram commute in $\TwoMon$. Uniqueness allows us to conclude the naturality (in $X$) of this structure from the naturality of the $2$-monoidal structure on the other arrows.\end{proof}

\begin{lemma}\label{lem:beta}
The identity functor ${\Pi_1 Z^n(\sing X\wedge \Delta^\bullet_+)}$ has a unique \textup(hence natural\textup) $2$-monoidal structure $\beta_X$ making the diagram
\[\xymatrix{
	\scrZ^n_{\varobar,\varominus}(X)\ar@{-->}[r]^{\beta_X}&\scrZ^n_+(X)\\
	\left(\Pi_1 \Omega^2 Z^{n-2}(\sing X \wedge \Delta^\bullet_+), \varobar,\varominus\right)\ar[r]_{\eqref{plus-varobar-varominus}}\ar[u]^{\sim}_{\eqref{double-cocycle}}&
	\left(\Pi_1 \Omega^2 Z^{n-2}(\sing X \wedge \Delta^\bullet_+), +, +\right)\ar[u]_{\Pi_1 (\psi\Omega\psi)}
}\]
commute as a diagram of $2$-monoidal categories and functors.
\end{lemma}

The right vertical map is $2$-monoidal since $\psi$ is linear. The proof of Lemma~\ref{lem:beta} is now immediate from Proposition~\ref{prop:equivalent-functors}~(i).\smallskip

Both categories $\scrZ^n_+(X), \scrZ^n(X)$ have the same objects $Z^n(X)$ and we let $\gamma_X$ be the identity on objects. To a morphism $f\in {Z^n( \sing X \wedge \Delta^1_+ )}$ in $\scrZ^n_+(X)$ from $d_1 f$ to $d_0 f$ we assign the class of the cochain $f / [\Delta^1] \in C^{n-1}(X;V) / \im(\delta)$.

\begin{lemma}\label{lem:gamma}
$\gamma_X\colon \scrZ^n_+(X) \rightarrow \scrZ^n(X)$ is a well-defined strict $2$-monoidal functor.
\end{lemma}
\begin{proof}
In the fundamental groupoid, a composition $f\circ g = h$ is `witnessed' by a $2$-simplex $\sigma \in Z^n(\sing X \wedge \Delta^2_+)$, meaning $\partial \sigma = g - h + f$. Hence \eqref{stokes-slant} implies that
\[
	(-1)^n\delta(\sigma/[\Delta^2]) = \sigma/\partial [\Delta^2] = g/[\Delta^1] - h/[\Delta^1] + f/[\Delta^1]
\]
is a coboundary, which proves that $\gamma_X$ is a functor. To show that $\gamma_X$ is well-defined, let $\sigma$ be a homotopy from $d_0\sigma = f$ to $d_1 \sigma = f'$ with $d_2\sigma = 0$. Then $(-1)^n\delta(\sigma / [\Delta^2]) = f/[\Delta^1] - f'/[\Delta^1] + 0$ exhibits the required coboundary. Since taking slant products is linear, $\gamma_X$ is strict $2$-monoidal.
\end{proof}

Combining Lemmas~\ref{lem:alpha}, \ref{lem:beta}, \ref{lem:gamma}, we define $\ch_X$ to be the composite  $2$-monoidal functor $\gamma_X\beta_X\alpha_X$. Explicitly, $\ch_X$ is given on objects and morphisms as follows:
\[
	\ch_X\colon \Map(X, E_n) \rightarrow \scrZ^n(X),\quad
	\begin{cases}
	\text{objects $f$:}& \ch(f)=f^*\iota_n,\\
	\text{morphisms $H\colon f\simeq g$:} & \ch(H) = H^*\iota_n / [\Delta^1].
	\end{cases}
\]
In particular, $\ch_X$ recovers \eqref{classical-chern} on isomorphism classes of objects. $\alpha_X, \beta_X, \gamma_X$ are natural in $X$, so this holds for $\ch_X$, too. This completes the proof.\qed

\section{Application to Differential Cohomology}\label{sec:application}

We begin by unravelling parts of Theorem~\ref{thm-refined-chern} into more elementary form. As shown in \cite{joyal1993braided}, there is an equivalence $\TwoMon \rightarrow \MonCat_\mathrm{braid}$ to braided monoidal categories. Hence we regard $\Map(X,E_n)$ as having just a single monoidal structure $\varobar$ and a natural braid (given by Theorem~\ref{prop:joyal-street}) and the functors $\ch_X$ from Theorem~\ref{thm-refined-chern} as having a natural braided monoidal structure $s$.\medskip

% Let $\alpha_n\colon\pr_1\varobar\pr_2\colon E_n\times E_n \rightarrow E_n$ for the projections $\pr_1, \pr_2 \in \Map(E_n\times E_n, E_n)$ (so $\alpha_n$ is loop composition). By naturality, the monoidal structure on $\Map(X,E_n)$ is given by $f\varobar g = \alpha_n\circ (f,g)$, which is also clear from the definition.

Fix the standard homotopies showing that ${\pi_0\Omega^2 E_{n+2}}$ is an abelian group (so the associator ${a=a_{\pr_1, \pr_2, \pr_3}}$ in ${\Map(E_n^{\times 3}, E_n)}$, braid ${s=s_{\pr_1, \pr_2}}$ in ${\Map(E_n^{\times 2}, E_n)}$, and unit constraint ${r=r_\id}$ in ${\Map(E_n,E_n)}$):% For example, $a=a_{\pr_1,\pr_2,\pr_3}$ in $\Map(E_n\times E_n\times E_n, E_n)$ and naturality implies $a_{f,g,h} = a\circ ((f,g,h)\times\id_{[0,1]})$ for $f,g,h$ in $\Map(X,E_n)$):
\begin{align*}
a: E_n^{\times 3}\times I &\rightarrow E_n, && \alpha_n\circ (\alpha_n\times \id)\simeq \alpha_n\circ(\id\times \alpha_n),\\
s: E_n^{\times 2}\times I &\rightarrow E_n, && \alpha_n\circ\mathrm{flip} \simeq \alpha_n,\\
r: E_n\times I &\rightarrow E_n, && \alpha_n\circ (\id,\const) \simeq \id.
\end{align*}
The monoidal structure $\varobar$ was induced by horizontal concatenation of loops:
\[
	\alpha_n\colon E_n\times E_n \approx \Omega^2 E_{n+2} \times \Omega^2 E_{n+2} \rightarrow \Omega^2 E_{n+2} \approx E_n.
\]
Then (either by direct inspection or using the naturality in $X$), the associativity and unit contraints $a,r$ as well as the braid $s$ on the categories $\Map(X,E_n)$ are given by post-composition with the above homotopies.

%It will be our convention that $\varobar$ denotes the monoidal structure that comes by writing $E_n\approx \Omega E_{n+1}$ as a one-fold loop space and $\varominus$ comes by writing it as a double loop space. Be warned however that this notation is potentially confusing because $\varobar$ on $\Pi_1 \Omega E_{n+1}$ in degree $n+1$ (inherited from $E_{n+1},\varobar$ in the notation of \eqref{eqn:halb-eq}) corresponds under the isomorphism (which is therefore strict monoidal) to the monoidal structure $\varominus$ of degree $n$ on $\Pi_1 E_n$.

\begin{theorem}\label{theorem-main-appl}
There exist reduced cochains $A_n \in C^{n-1}(E_n\times E_n;V)$ satisfying
\begin{equation}\label{eqn:choice-an}
	\delta A_n = \pr_1^*\iota_n + \pr_2^*\iota_n - \alpha_n^*\iota_n
\end{equation}
and coherent in the sense that \textup(`$\equiv$' means `up to coboundary'\textup)
\begin{align*}
\pr_{12}^*A_n + (\alpha_n\times 1)^*A_n &\equiv
\pr_{23}^*A_n+(1\times\alpha_n)^*A_n + \ch(a), &&\text{associative}\\
\mathrm{flip}^* A_n &\equiv A_n + \ch(s), &&\text{commutative}\\
(\id_{E_n},\const)^* A_n &\equiv \ch(r). &&\text{unit}
\end{align*}
%Here, $\pr_{12}$ denotes the projection $E_n^{\times 3} \rightarrow E_n^{\times 2}$ onto the first and second factor.
\textup(Recall that $\ch(h)=h^*\iota_n / [0,1] = \smallint_0^1 h^*\iota_n$ for morphisms/homotopies $h$.\textup)
\end{theorem}

\begin{proof}
The data of a monoidal functor
\[
	\ch_X: (\mathrm{Map}(X,E_n),\varobar) \rightarrow (\scrZ^n(X),+).
\]
includes morphisms relating $\varobar,+$; that is, elements $\ch^\varobar_{c,d} \in C^{n-1}(X) / \im(\delta)$ with
\begin{equation}\label{eqn:how-An-arises}
	\delta\ch^\varobar_{c,d} = \ch(f\varobar g)-\ch(f)-\ch(g).
\end{equation}
%For the units we have $\ch^\varobar_\mathrm{const} = 0$, which induces an isomorphism $\delta 0 = \ch(\mathrm{const})-0$.
Naturality gives commutative diagrams of braided monoidal functors
\[\xymatrix@C=1cm{
	(\mathrm{Map}(X,E_n),\varobar)\ar[r]^-\ch&(\scrZ^n(X),+)\\
	(\mathrm{Map}(Y,E_n),\varobar)\ar[u]^{\Map(f,E_n)}\ar[r]_-\ch&(\scrZ^n(Y),+)\ar[u]_{\scrZ^n(f)}
}\]
which means
%By definition of composition in $\scrZ^n(X)$ and since the vertical functors are strict this means
\begin{equation}\label{eqn:nat-of-ch-an}
	f^*\ch^\varobar_{c,d} = \ch^\varobar_{f^*c, f^*d}\quad\text{in}\quad C^{n-1}(X)/\im(\delta),\qquad c,d: Y\rightarrow E_n.
\end{equation}
A braided monoidal functor has to satisfy various coherence conditions:
\[
\xymatrix{
\ch( (f\varobar g)\varobar h)\ar[d]_{\ch(a)}\ar[r]^{\ch_{f\varobar g, h}} & \ch(f\varobar g)+\ch(h)\ar[r]^-{\ch_{f,g}} & (\ch(f)+\ch(g))+\ch(h)\ar@{=}[d]\\
\ch( f\varobar (g\varobar h))\ar[r]_{\ch_{f,g\varobar h}} & \ch(f)+\ch(g\varobar h)\ar[r]_-{\ch_{g,h}} & \ch(f)+(\ch(g)+\ch(h)).
}\]
\[
\begin{aligned}
&\xymatrix{
\ch(f)+\ch(\const)\ar@{=}[r]&\ch(f)+0\ar@{=}[d]\\
\ch(f\varobar\const)\ar[u]^{\ch^\varobar_{f,\const}}\ar[r]_{\ch(r)}& \ch(f)
}&
\xymatrix{
	\ch(f)+\ch(g)\ar@{=}[r]&\ch(g)+\ch(f)\\
	\ch(f\varobar g)\ar[u]^{\ch_{f,g}}\ar[r]_{\ch(s)} & \ch(g\varobar f)\ar[u]_{\ch_{g,f}}
}
\end{aligned}
\]
Set $c=\pr_1, d=\pr_2: E_n\times E_n\rightarrow E_n$ in \eqref{eqn:how-An-arises} to define
\begin{equation}\label{eqn:inproof-def-an}
	A_n = \ch^\varobar_{\pr_1, \pr_2}.
\end{equation}
With this notation, the commutativity of the first coherence diagram reads
\begin{equation}\label{eqn:explicit-assoc-diag-an}
	\ch^\varobar_{g,h} + \ch^\varobar_{f,g\varobar h} +\ch(a) \equiv \ch^\varobar_{f,g} + \ch^\varobar_{f\varobar g,h}\quad\text{in}\quad C^{n-1}(X)/\im(\delta).
\end{equation}
If we set $f=\pr_1, g=\pr_2, h=\pr_3: E_n\times E_n\times E_n \rightarrow E_n$, naturality \eqref{eqn:nat-of-ch-an} asserts
\begin{align*}
	f=(1\times\alpha_n)^*\pr_1, g\varobar h = (1\times \alpha_n)^*\pr_2 &\Rightarrow \ch^\varobar_{f,g\varobar h} \equiv (1\times \alpha_n)^*\ch^\varobar_{\pr_1,\pr_2},\\
	g=\pr_{23}^*\pr_1, h=\pr_{23}^*\pr_2 &\Rightarrow \ch_{g,h}^\varobar \equiv \pr_{23}^*\ch^\varobar_{\pr_1,\pr_2},\\
	f\varobar g=(\alpha_n\times 1)^*\pr_1, h=(\alpha\times 1)^*\pr_2 &\Rightarrow \ch^\varobar_{f\varobar g, h} \equiv (\alpha_n\times 1)^*\ch^\varobar_{\pr_1,\pr_2},\\
	f=\pr_{12}^*\pr_1, g=\pr_{12}^*\pr_2 &\Rightarrow \ch_{f,g}^\varobar \equiv \pr_{12}^*\ch^\varobar_{\pr_1,\pr_2}.
\end{align*}
Inserting these equalities and \eqref{eqn:inproof-def-an} into \eqref{eqn:explicit-assoc-diag-an} gives
\[
	\pr_{23}^*A_n + (1\times\alpha_n)^* A_n + \ch(a) \equiv \pr_{12}^*A_n + (\alpha_n\times1)^*A_n%\;\;\text{in}\;\; C^{n-1}(X)/\im(\delta).
\]
Similarly, the second coherence diagram for $f=\pr_1$ asserts
\[
\ch(r)\equiv \ch^\varobar_{\pr_1, \const} \overset{\eqref{eqn:nat-of-ch-an}}{\equiv}(\id,\const)^*\ch_{\pr_1,\pr_2}=(\id,\const)^*A_n
\]
The third diagram for $f=\pr_1, g=\pr_2$ says, using naturality \eqref{eqn:nat-of-ch-an} for $\pr_2 = \mathrm{flip}^*\pr_1, \pr_1 = \mathrm{flip}^*\pr_2$:
\[
	\mathrm{flip}^*A_n + \ch(s)  \equiv \mathrm{flip}^*\ch_{f,g} + \ch(s)\equiv\ch_{g,f} + \ch(s) \equiv \ch_{f,g} = A_n\qedhere
\]
\end{proof}

Theorem~\ref{theorem-main-appl} contains exactly the coherence conditions needed to prove that \eqref{add-dif-cocycles} gives an abelian group structure. The key observation is (see~\cite[3.10,~3.13]{dr-arbeit}):

\begin{proposition}\label{additioncentral}
\textup{(i)} Given a homotopy ${C\colon c_0 \simeq c_1}$ of maps, a form ${\omega \in \Omega^n_\mathrm{cl}(M;V)}$, and cochain $h \in C^{n-1}(M; V)$ with $\delta h = \omega - c_0^*\iota_n$, we have an equivalence% \textup(similarly in the relative case $N\subset M$\textup)
\[
	(c_0, \omega, h) \sim (c_1,\omega,h-\ch(C)).
\]
\textup{(ii)} For a cocycle $(c,\omega,h)$ and $g \in C^{n-2}(M;V)$ we have $(c, \omega, h) \sim (c, \omega, h+\delta g)$.
\end{proposition}

\subsection{Proof of Theorem~\ref{mainappladd}}\label{proof-thm2}
Applying part (i) to the homotopies $a, r, s$ above and then part (ii) to the coherence equations in Theorem~\ref{theorem-main-appl} shows that \eqref{add-dif-cocycles} descends to an associative, unital, and commutative operation on equivalence classes.

It remains to show that we have inverses. Pick maps $\nu_n\colon E_n\rightarrow E_n$ representing negation in $E$-cohomology and a homotopy $h\colon \id_n\varobar \nu_n = \alpha_n(\id_n,\nu_n) \simeq \const$. For
\[
	N_n = \ch(h) - (\id,\nu_n)^*A_n
\]
we have $\delta N_n = -\iota_n - \nu_n\iota_n$. Applying Proposition~\ref{additioncentral} to the homotopy $h$ then shows that $(c,\omega,h)+(\nu_n \circ c, -\omega, -h + c^*N_n)$ is equivalent to zero.
\qed

\appendix
\section{}\label{appendix}

From \cite{kelly1974doctrinal} (or see \cite[Appendix~A]{dr-arbeit}) we recall the following well-known fact:

\begin{theorem}[Doctrinal Adjunction]\label{doctrinal}
Suppose $(F,G,\varepsilon, \eta)$ is an adjoint equivalence in which $F\colon\scrC \rightarrow \scrD$ is a monoidal functor. Then there exists a unique monoidal structure on $G$ that makes $(F,G,\varepsilon, \eta)$ a monoidal adjoint equivalence.
\end{theorem}

%In case $F$ is a $2$-monoidal functor, Theorem~\ref{doctrinal} may be applied to both monoidal structures $\varobar, \varominus$ individually. Then $G$ is automatically $2$-monoidal: the commutativity of \eqref{axiom-2-mon-functor} for $G$ may be checked after applying $F$ (being faithful) and the counit $\varepsilon\colon FG\rightarrow 1_\scrD$ identifies this diagram with the corresponding diagram \eqref{axiom-2-mon-functor} for $F$.

\begin{lemma}\label{prop:lift-of-mon-functors-section}
Suppose $G\circ F = H$ are functors, where $F,H$ are monoidal and $F$ is an equivalence. There exists a unique monoidal structure on $G$ so that $G\circ F = H$ as monoidal functors. \textup(similarly, if $G$ is an equivalence, $G,H$ monoidal.\textup)
\end{lemma}
\begin{proof}
If a monoidal structure on $G$ exists, we must have
$
	H_{C_1, C_2} = G(F_{C_1,C_2})\circ G_{FC_1, FC_2}$ and $H_1 = G(F_1)\circ G_1$. Hence $G_{E_1,E_2}$ is determined on the image of $F$. For general $E_1, E_2$ pick isomorphisms $\varphi_i\colon E_i\rightarrow FC_i$. Naturality gives $G_{FC_1, FC_2}\circ (G\varphi_1\otimes G\varphi_2) = G(\varphi_1\otimes\varphi_2) \circ G_{E_1,E_2}$, so $G_{E_1, E_2}$ is uniquely determined. To prove existence, place $F$ in an adjoint equivalence $(F,R,\varepsilon,\eta)$ which, by doctrinal adjunction, may viewed as a monoidal adjunction. We have a natural isomorphism
$
	G\varepsilon\colon HR=GFR \xlongrightarrow{\cong} G.
$
There is a unique monoidal structure on $G$ making $G\varepsilon$ a monoidal transformation (monoidal structures on functors may be uniquely transported along natural isomorphisms). The composition of a monoidal transformation with a monoidal functor is again monoidal. Therefore $H\eta\colon H \rightarrow HRF$, $G\varepsilon F\colon HRF\rightarrow GF$ are monoidal transformations which compose to the identity, by the zig-zag identities for $(F,R,\varepsilon,\eta)$. But this just means $H=GF$ as monoidal functors.
\end{proof}

\subsection{Proof of Proposition~\ref{prop:equivalent-functors}}

We prove only (i) since (ii) follows by a dual argument. Placing $F$ in an adjoint equivalence $(F,R,\varepsilon,\eta)$, Lemma~\ref{prop:lift-of-mon-functors-section} gives us two monoidal structures on $G$ satisfying ${G^\varobar \circ F^\varobar = H^\varobar}$ and ${G^\varominus \circ F^\varominus = H^\varominus}$.
Since $H, F$ preserve $\zeta$, we have (omitting the vertical maps in \eqref{axiom-2-mon-functor} from the notation)
\[
	\zeta_{HA_1,HA_2,HA_3,HA_4}^\scrE = H\zeta_{A_1,A_2,A_3,A_4}^\scrC = GF(\zeta_{A_1,A_2,A_3,A_4}^\scrC) = G(\zeta_{FA_1,FA_2,FA_3,FA_4}^\scrD).
\]
Hence $G$ preserves the interchange on the image of $F$. In general, we may pick isomorphisms $X_i \cong FA_i$ ($F$ is essentially surjective). Naturality of the interchange then implies that $G$ preserves $\zeta_{X_1,X_2,X_3,X_4}^\scrD$. Hence $G$ is $2$-monoidal.\qed\newline

We will use the following well-known fact (see \cite[Appendix~A]{dr-arbeit}):

\begin{proposition}
Let $(F,G,\varepsilon,\eta)$ be an adjoint equivalence from $\scrC$ to $\scrD$ and assume that $\scrD$ is a monoidal category. Then there exists a monoidal structure on $\scrC$ making $(F,G,\varepsilon,\eta)$ a monoidal adjoint equivalence.\\
\textup{
(for example, one can set $C_1\otimes C_2 = G(FC_1\otimes FC_2)$, $I_\scrC = GI_\scrD$.)
}
\end{proposition}

%The monoidal structure on $F$ will be the counit ${\varepsilon_{FC_1\otimes FC_2}^{-1}\colon FC_1\otimes FC_2 \rightarrow F(C_1\otimes C_2)}$ and on $G$ will be ${G(\varepsilon_{D_1}\otimes \varepsilon_{D_2})\colon GD_1\otimes GD_2 \rightarrow G(D_1\otimes D_2)}$. The bijectivity of $F$ on $\Hom$-sets gives a unique way to define the associativity and unit constraints on $\scrC$ so as to make $F$ preserve them (so $F$ will be a monoidal functor). Mac~Lane's pentagon and the triangle identities may checked after an applying the faithful $F$, where they reduce to the corresponding properties in $\scrD$.\medskip

%The structure on $\scrD$ is unique up to equivalence, but there is no canonical choice.

Applied to both monoidal structures of a $2$-monoidal category $\scrD$, we obtain two monoidal adjunctions $(F^\varobar, G^\varobar, \varepsilon, \eta)$ and $(F^\varominus, G^\varominus, \varepsilon, \eta)$. As \eqref{axiom-2-mon-functor} is a diagram of isomorphisms and $F$ is bijective on $\Hom$-sets, there is a unique interchange $\zeta^\scrC$ making $F$ a $2$-monoidal functor. The verification of the diagrams \eqref{compatiblity-zeta} expressing the compatibility of $\zeta^\scrC$ with associativity and unit constraints can, $F$ being faithful, be reduced to the corresponding properties of $\zeta^\scrD$.

\subsection{Proof of Proposition~\ref{prop:equivalent-stuff}}

Placing each of the equivalences $F_i\colon \scrC_i \rightarrow \scrD_i$ into an adjoint equivalence $(F_i, G_i, \varepsilon_i, \eta_i)$ \cite[Proposition~1.1.2]{leinster2004higher}, the preceeding remark gives $2$-monoidal structures $\hat{\scrC}_i$ on $\scrC_i$ and $\hat{F}_i$ on every $F_i$. It remains to define $2$-monoidal structures on every $\scrC(f)\colon \scrC_i \rightarrow \scrC_j$ for morphisms $f\colon i\rightarrow j$  in $\scrI$. The naturality of $\hat{F}_i$ means that
\[\xymatrix{
	\hat{\scrC}_i\ar[r]^{\scrC(f)}\ar[d]_{\hat{F_i}}&\hat{\scrC}_j\ar[d]^{\hat{F_j}}\\
	\scrD_i\ar[r]_{\scrD(f)}&\scrD_j
}\]
should commute as a diagram in $\TwoMon$. Proposition~\ref{prop:equivalent-functors} implies that there is a unique such $2$-monoidal structure on $\scrC(f)$. Uniqueness implies functoriality.\qed

\bibliography{references}
\bibliographystyle{plain}

\end{document}